\documentclass[11pt]{article}
\usepackage[a4paper,margin=1in]{geometry}
\usepackage{amsthm,xcolor}
\usepackage{amsmath,mathtools}
\usepackage{amssymb}
\usepackage{appendix}
\usepackage{graphicx} %
\usepackage{placeins}
\usepackage{abstract}
\usepackage{enumerate}
\newtheorem{definition}{Definition}[section]
\newtheorem{theorem}{Theorem}[section]  
\newtheorem{lemma}[theorem]{Lemma}     
\newtheorem{proposition}[theorem]{Proposition}  
\newcommand{\CA}{C_0(\mathbb{A})}
\newtheorem{example}[theorem]{Example}
\newcommand{\sectwo}{Section 2}

\newcommand{\Z}{\mathbb{Z}}
\newcommand{\Q}{\mathbb{Q}}
\newcommand{\R}{\mathbb{R}}
\newcommand{\T}{\mathbb{T}}
\newcommand{\A}{\mathbb{A}}

\newcommand{\cstar}{C^{*}}
\newcommand{\what}[1]{\widehat{#1}}
\newcommand{\id}{\mathrm{id}}

\newcommand{\Hom}{\mathrm{Hom}}

\usepackage[colorlinks=true, linkcolor=blue, urlcolor=red]{hyperref}
\usepackage[style=numeric,backend=biber,sorting=none,maxnames=99]{biblatex} 
\newtheoremstyle{remark}
    {3pt}            
    {3pt}            
    {\normalfont}    
    {}               
    {\bfseries\color{blue}} 
    {.}              
    { }              
    {}               

\theoremstyle{remark}
\newtheorem{remark}{Remark}[section] 

\title{$K$-Theory of Adelic and Rational Group $C^*$-algebras via Generalized Winding Numbers}

\author{Hang Wang\footnote{Research Center of Operator Algebras, East China Normal University, Shanghai, China. \newline Email: wanghang@math.ecnu.edu.cn} ~~and~ Wenqing Wu\footnote{Research Center of Operator Algebras, East China Normal University, Shanghai, China. \newline Email: wuwenqingsz@qq.com} }

\begin{document}
\maketitle
\begin{abstract}
    We establish a two-part framework for analyzing homotopy equivalence in periodic adelic functions. First, we introduce the concept of pre-periodic functions and define their homotopy invariant through the construction of a generalized winding number. Subsequently, we establish a fundamental correspondence between periodic adelic functions and pre-periodic functions. By extending the generalized winding number to periodic adelic functions, we demonstrate that this invariant completely characterizes homotopy equivalence classes within the space of periodic adelic functions. Building on this classification, we further obtain an explicit description of the $K_{1}$-group of the rational group $C^\ast$-algebra, $K_{1}(C^{*}(\mathbb{Q}))$. Finally, we further employed a similar strategy to derive the K-theory of \(K_1(C^{\ast}(\A))\).
\end{abstract}
\section{Introduction}
The study of $C^\ast$-algebras related to number theory has attracted significant interest in the intersection of operator algebras and arithmetic. Some works consider $C^*$-algebras associated with integral domains or rings of integers in number fields, involving their crossed product structures over finite and infinite adele spaces\cite{ref1,ref2}; in addition, other studies investigate crossed product $C^\ast$-algebras associated with adele rings \cite{ref3,ref4}. The computation of $K$-theory for group $C^*$-algebras associated with adelic structures is a fundamental problem at the intersection of functional analysis and number theory. Investigations into $p$-adic and adelic $C^*$-algebras, exemplified by works such as \cite{MR3300223} and \cite{MR883498}, have predominantly focused on the structure of $C^*$-algebras attached to adelic general linear groups.
While these papers employ advanced techniques from noncommutative geometry and higher algebraic $K$-theory, we instead focus directly on the $C^*$-algebra of the adele ring $\mathbb{A}$ itself and develop tools that bridge these abstract methods with more concrete harmonic-analytic techniques, providing a complementary perspective.
In particular, this paper develops elementary methods to compute the $K$-theory of $C^*(\mathbb{A})$ and $C^*(\mathbb{Q})$ while establishing new connections between adelic analysis and topological invariants.

The continuous functions on $\mathbb{A}$, although fundamental to its analytical structure, have received limited direct attention due to the prevailing technical approach of approximating them by adelic smooth functions. However, if we aim to compute the K-theory via the isomorphism between \( C^*(\mathbb{Q}) \) and \( C_0(\A/\mathbb{Q}) \), a direct analysis of adelic continuous functions becomes necessary. We describe certain topological properties of adelic $\mathbb{Q}$-periodic continuous functions by introducing \textit{pre-periodic functions} – a class of real-valued functions capturing the essential approximate periodicity of their adelic counterparts. Formally, for $f \in C(\mathbb{R}, S^1)$, we define pre-periodicity through the condition:
$$
\forall \varepsilon > 0,\ \exists N\in \mathbb{Z}^+\ \text{such that}\ \forall k \in \mathbb{Z},\ x \in \mathbb{R}:\ |f(x + kN) - f(x)| < \varepsilon.
$$
This approximate periodicity  enables the construction of a $\mathbb{Q}$-valued homotopy invariant, the \textit{generalized winding number}, defined for pre-periodic functions as:
$$
\lim_{x\to+\infty}\frac{\tilde{f}(x)-\tilde{f}(-x)}{2x},
$$
where $\tilde{f}$ denotes a lift of $f$ to $\mathbb{R}$. For adelic functions, we fix the finite adele component and compute this invariant for the resulting real-valued function, thereby linking local analytic behavior to global topological properties.

Our methodology distinguishes itself from previous work through its simplicity. We utilize elementary tools: Gelfand transforms for commutative $C^*$-algebras, function lifting techniques in covering spaces, and the Chinese Remainder Theorem for adelic approximations. This approach not only makes the $K$-theoretic computations more transparent but also reveals hidden geometric structures in the adelic function spaces. Specifically, the pre-periodicity condition allows decomposition of adelic functions into local components while preserving global homotopical information through the winding number invariant.

The main result of this paper can be summarized in the following theorem:
\begin{theorem}\label{thm:main}
The $K$-theory of the group $C^*$-algebras associated with the adele ring $\A$ 
and its lattice $\Q\subset \A$ can be described as follows:
\begin{enumerate}[(a)]

\item For the rational lattice $\Q\subset \A$,
\[
K_0\bigl(C^*(\Q)\bigr) \;\cong\; \Z, 
\qquad
K_1\bigl(C^*(\Q)\bigr) \;\cong\; \Q,
\]
as stated in Remark \ref{K0Q} and Theorem \ref{K1Q}.
\item For the adele ring $\A$,
\[
K_0\bigl(C^*(\A)\bigr) \;\cong\; 0, 
\qquad
K_1\bigl(C^*(\A)\bigr) \;\cong\; C_0(\A_f,\Z),
\]
see Remark \ref{K0A} and Theorem \ref{K1A}.
\end{enumerate}
\end{theorem}

A key motivation for the explicit computation of the pair $(\mathbb{A}, \mathbb{Q})$ is that it provides a crucial foundation for our future investigations into the $K$-theoretic trace formula \cite{MR4200263} in the adelic setting, as well as the related developments in higher index theory.

The paper is organized into the following sections. Section 2 begins by employing the Gelfand transform to recast the study of the group $C^*$-algebra $C^*(\mathbb{A})$ and $C^*( \mathbb{Q})$ in terms of continuous functions on the quotient space $\mathbb{A}/\mathbb{Q}$. Section 3 includes, for completeness, an alternative perspective on the $K$-theory of $C^*(\Q)$, 
making use of its inductive limit description together with the continuity of $K$-theory. Section 4 then develops two pivotal constructs based on the framework established in Section 2: pre-periodic functions on $\mathbb{R}$, systematically characterized by a homotopy invariant termed the \textit{generalized winding number}. These theoretical advances naturally lead to subsequent applications in Sections 5 and 6, where we employ the established framework to achieve a complete classification of homotopy equivalence classes for unitary elements in both $C(\mathbb{A}/\mathbb{Q})$ and $C_0(\mathbb{A})$.

\section*{Acknowledgments} 
The authors gratefully acknowledge support from the National Natural Science Foundation of China (Grant No. 12271165), and from the Science and Technology Commission of Shanghai Municipality (Grant Nos. 23JC1401900 and 22DZ2229014).
We are also indebted to Haluk Şengün for his guidance in providing relevant references in number theory and for valuable discussions regarding the alternative proof presented in Section~\ref{sec3}.

\section{The Pontryagin dual of \texorpdfstring{\(\A\)}{A} and \texorpdfstring{\(\mathbb{Q}\)}{Q}}

This section introduces the duality correspondence through the Gelfand transform, mapping group $C^*$-algebras of the adele ring $\mathbb{A}$ and its discrete subgroup $\mathbb{Q}$ into the continuous function space on their Pontryagin duals. While the duality theory for local fields has been thoroughly established in \cite{MR1344916}, a more nuanced investigation into the adele ring remains imperative for further advancements in this area. The subsequent analysis builds fundamentally on the self-duality of $\A/\mathbb{Q}$.

\begin{definition}
   The space \( L^1(\mathbb{A}) \) becomes a \( * \)-algebra when equipped with the following operation:
   \begin{align*}
       f* g(t)&=\int f(s)g(t-s)ds,\\
       f^*(t)&=\overline{f(-t)}.
   \end{align*}
Define a $C^*$-norm on $L^1(\A)$ by
\[
\|f\|=\sup\{\|\pi(f)\|:\pi\text{ is a } *-\text{representation of }L^1(\A)\}.
\]
The \textbf{group C\(^*\)-algebra} of $\A$ is the completion of $L^1(\A)$ in this norm. Similarly, C$^*(\mathbb{Q})$ can also be defined.
\end{definition}
The following propositions provided in \cite{cem} can simplify the calculation of K-theory in C$^*$-algebras:
\begin{proposition}
    If $G$ is an abelian group, then
    \[
    \text{C}^*(G)=\text{C}_0(\hat{G}).
    \]
\end{proposition}
It shows that the Pontryagin dual is an effective tool for solving the problem. The proof of the Pontryagin duality for the Adele ring \(\mathbb{A}\) centers on decomposing global questions into harmonic analysis over local fields and unifying local duality properties via the restricted product structure. A primary consideration involves examining the Pontryagin dual of \(\mathbb{Q}_p\) (the field of \(p\)-adic numbers).

\begin{lemma}\label{Q_pdual}
    \(\mathbb{Q}_p\) is self-dual.
\end{lemma}
\begin{proof}
First we define a map $\mathcal{Q}$ from $\mathbb{Q_p}$ to $\hat{\mathbb{Q}}_p$:
\begin{equation}
    y_p\rightarrow e_{y_p}, \quad \forall y_p\in \mathbb{Q}_p,
\end{equation}
where \(e_{y_p}(x_p):=e^{2\pi i\cdot \{y_p\cdot x_p\}}\), here $x_p\in \mathbb{Q}_p, \{y_p\cdot x_p\}=y_p\cdot x_p-\lfloor y_p\cdot x_p\rfloor.$ This map is clearly an injective homomorphism, and it remains to show that $\mathcal{Q}$ is also  surjective.

Given a character \(\chi: \mathbb{Q}_p \to S^1\), which naturally satisfies \(\chi(0) = 1\), according to its continuity, for any integer \(\varepsilon >0\), we have \(\delta>0 \) such that: 
\begin{equation}\label{lim}
    \lvert\chi(x_p)-1\rvert<0,\ \forall \lvert x_p\rvert_p<\delta.
\end{equation}
so there exists a positive integer \( n_{8}\) such that the following condition is satisfied:
\begin{align}
    \lvert\chi(k\cdot p^n)-1\rvert  \leq \frac{1}{8}, \forall n\geq n_8.
\end{align}
Hence, $\chi(p^{n_8})$ generates the following subgroup of \(S^1\) with diameter not greater than $\frac{1}{4}$:
\[
\{\chi(k\cdot p^{n_8})|k\in \mathbb{Z}\},
\]
implies that
\[
\chi( p^{n_8})=1.
\] 
We see that there exists an \(n_0\in \mathbb{Z}^+\), such that for every \(n\geq n_0\),  \(\chi (p^n)=1\).

Above result allow us to choose an integer $n_m$ such that:
\begin{equation}
    \chi(p^{n})= 1, \forall n>n_m, 
\end{equation}
while
\[
\chi(p^{n_m})\neq 1.
\]
Hence, there exists an integer $0\leq k_m\leq p-1$ such that:
\begin{equation}
    \chi(p^{n_m})=e^{2\pi i\frac{k_m}{p} }.
\end{equation}
Similarly,
\begin{align*}
    \chi(p^{n_m-1})&=e^{2\pi i (k_m+p\cdot k_{m-1})/{p^2}},\\
    \chi(p^{n_m-2})&=e^{2\pi i(k_m+p\cdot k_{m-1}+p^2\cdot k_{m-2})/{p^3}},\\
    &\vdots\\
    \chi(p^{n_m-l})&=e^{2\pi i(\sum_{j=0}^lp^j\cdot k_{m-j})/{p^{l+1}}},\\
    &\vdots
\end{align*}
where for any integer $j\geq 0$, $k_{m-j}$ is an integer between \(0\) and \(p-1\).

Any $p$-adic number $x_p$ can be expressed in the form  \(\sum\limits_{n=n_0}^{\infty}a_n\cdot p^n\), where $n_0\in \mathbb{Z}$, $a_n$ is an integer between \(0\) and \(p-1\), then
\begin{align}\label{cha}
    \chi(x_p)&=\chi(\sum\limits_{n=n_0}^{\infty}a_n\cdot p^n)\\
    &=\prod\limits_{n=n_0}^{\infty} \chi(p^n)^{a_n}=\prod_{n=n_0}^{\infty}e^{2\pi i\left(\sum\limits_{j=-1}^{\infty}k_{m-1-j}\cdot p^{j-n_m}\right)\cdot a_n\cdot p^n}\\
    &=e^{2\pi i\cdot(\sum_{j=-1}^{\infty}p^{j-n_m}\cdot k_{m-1-j})\cdot x_p},
\end{align}
we have
\begin{equation}
    \chi=\mathcal{Q}\left(\sum_{j=-1}^{\infty}p^{j-n_m}\cdot k_{m-1-j}\right).
\end{equation}
After showing that the mapping \(\mathcal{Q}\) is surjective, the isomorphism between $\mathbb{Q}_p$ and $\hat{\mathbb{Q}}_p$ is obtained.
\end{proof}
Furthermore, the fact that the additive group of $\mathbb{R}$ is self-dual under Pontryagin duality is a straightforward result in harmonic analysis. Building on this foundation, we can naturally extend our investigation to characterize the Pontryagin dual of the adele ring \( \mathbb{A} \). 
\begin{theorem}
        The adele ring $\A$ is self-dual.
\end{theorem}
\begin{proof}
Since both \(\mathbb{Q}_p\) and $\mathbb{R}$ embed into the adele ring \(\mathbb{A}\), any character $\chi$ on \(\mathbb{A}\) admits the following decomposition: 
\[
\chi(\{x_\infty,x_2,x_3,\dots\})=e^{2\pi i\cdot y_\infty\cdot x_\infty}\times\prod\limits_{p\ \text{prime}}e^{-2\pi i\cdot \{y_p\cdot x_p\}}.
\]
    where $y_\infty \in \mathbb{R},\ y_p\in \mathbb{Q}_p$ . Thus we can construct a map $\mathcal{A}: \A\rightarrow\hat{\A} $ defined by:
    \[
    y\mapsto e_y, \forall y=\{y_{\infty},y_2,y_3,\dots\}\in \A.
    \]
    The $e_y$ here is a map from $\A $ to $S^1$ :
    \[ e_y(x)=\chi(\{x_\infty,x_2,x_3,\dots\})=e^{2\pi i\cdot y_\infty\cdot x_\infty}\times\prod\limits_{p\ \text{prime}}e^{-2\pi i\cdot \{y_p\cdot x_p\}},\ \forall x =\{x_\infty,x_2,x_3,\dots\}\in \A.\]
    To prove that $\mathcal{A}$ is an isomorphism, it suffices to demonstrate that $\mathcal{A}$ is surjective, which is equivalent to for any 
    \[
    \chi(\{x_\infty,x_2,x_3,\dots\})=e^{2\pi i\cdot y_\infty\cdot x_\infty}\times\prod\limits_{p\ \text{prime}}e^{-2\pi i\cdot \{y_p\cdot x_p\}},
    \]
    only finitely many primes \textit{p} satisfy \(y_p\notin \mathbb{Z}_p\), i.e.\(\{y_{\infty},y_2,y_3,\dots\}\in \A\).
    We claim that there exists an integer $n_0$ such that for every prime number $p\geq n_0$, \[
    \chi(\{0_\infty,0_2,0_3,\dots, \mathbb{Z}_p,\dots\})=0.
    \]
    Consider the neighborhood basis of $0$ consisting of \textbf{countable  open sets} of the form:
\[
B_n = \left( -\frac{1}{n},\, \frac{1}{n} \right) \times \prod_{p \leq n} p^{n}\mathbb{Z}_p \times \prod_{p > n} \mathbb{Z}_p \quad (n \in \mathbb{N}). 
\]
Because $\chi$ is continuous, there is an $n_0\in\mathbb{N}$ such that
\[
\lvert\chi(B_{n_0})-1\rvert\leq 0.1.
\]
We assume that the n-th prime number \( p_n \) is greater than \( n_0 \), then we define a subset $R_{p_n}$of $B_n$:
\[
\{\{\underbrace{0_\infty,0_2,\dots,0_{p_{n-1}}}_{\text{all zero}},\underbrace{z_{p_n}}_{p\text{-adic integer}}, \underbrace{0_{p_{n+1},\dots}}_{\text{all zero}}\}\},
\]
which is an additive group generated by
\[
\{0_\infty,0_2,\dots,0_{p_{n-1}},1_{p_n}, 0_{p_{n+1},\dots}\}.
\]
Similarly, $\chi(R_{p_n})$ is a subgroup of $S^1$ with diameter not greater than $0.2$, so
\[
1=\chi(\{0_\infty,0_2,\dots,0_{p_{n-1}},1_{p_n}, 0_{p_{n+1},\dots}\})=e^{-2\pi i\cdot \{y_{p_n}\}},
\]
implies that \(y_{p_n}\in \mathbb{Z}_{p_n}\). Due to the arbitrariness of $p_n$, we conclude that only finitely many \( y_p \) do not belong to \( \mathbb{Z}_p \). Naturally, the sequence
 \[
 y=\{y_\infty,y_2,y_3,y_5,y_7,\dots\}\in \A,
 \]
 which proves that $\mathcal{A}$ is surjective. We finally obtained that \(\mathcal{A}\) is an isomorphism between $\mathbb{A}$ and $\hat{\mathbb{A}}$.
 \end{proof}
Since \( \mathbb{Q} \) is a discrete lattice in \( \mathbb{A} \), the characters of \(\mathbb{Q} \) can be extended to characters of \(\mathbb{A}\), so it suffices to focus on the restriction of  \(\hat{\mathbb{A}}\) to $\mathbb{Q}$.
Let $y\in \mathbb{A}$ and for any \(\alpha\in \mathbb{Q}\),
\[
e_{y}(\alpha)=e^{2\pi i\cdot y_\infty\cdot \alpha_\infty}\times\prod\limits_{p\ \text{prime}}e^{-2\pi i\cdot \{y_p\cdot \alpha_p\}}=1,
\]
where \(\{\alpha_\infty,\alpha_2,\alpha_3,\dots\}\) is the embedding of \(\alpha\) in \(\A\).
There are only finitely many $y_p\notin \mathbb{Z}_p$, so we can find a non-zero integer $N_2^y$ such that:
\[
N^y_2y_p\in \mathbb{Z}_p,\ \forall p\neq2.
\]
Then for any $k$ belongs to \(\mathbb{Z}\), 
\[
e_y(2^k\cdot N_2^y)=e^{2\pi i\cdot y_\infty\cdot 2^k\cdot N_2^y}\cdot e^{-2\pi i\cdot \{y_2\cdot 2^k\cdot N_2^y\}}=e^{2\pi i\cdot \{(y_\infty-y_2)\cdot 2^k\cdot N_2^y\}}=1, \
\]
implies that
\[
y_\infty=y_2.
\]
However, $\mathbb{R}\cap\mathbb{Q}_2=\mathbb{Q}$, which shows \(y_\infty,y_2\in \mathbb{Q}.\)

Similarly, we can conclude that :
\[
y_\infty=y_p.
\]
The variable \( y \) is of the form:
\[
\{y_\infty,(y_\infty)_2,(y_\infty)_3,(y_\infty)_5,\dots\},\ y_\infty\in \mathbb{Q},
\]
which implies that each \( y\in \A \) satisfying \( e_y(\alpha) = 1,\ \forall\alpha\in\mathbb{Q} \) is in \textbf{one-to-one correspondence} with a rational number. This establishes the following lemma:
\begin{lemma}
    The Pontryagin dual of  \(\mathbb{Q}\) as a lattice in $\A$ is $\A/\mathbb{Q}$.
\end{lemma}
These results allow us to replace C$^*(\mathbb{Q})$, C$^*(\mathbb{A})$ with C$(\A/\mathbb{Q})$, C$_0(\A)$ for our investigation in the following section.

\section{A standard argument via inductive limits}\label{sec3}

There is a direct way to compute the K-theory of $C^*(\Q)$ using the continuity of K-theory for inductive limits. We recall this argument for completeness, and then proceed to present our elementary computation, which provides additional insight into the structure of $K_*(C^*(\Q))$.

Let $\mathcal{N}$ be the directed set of positive integers ordered by divisibility: $N\preceq M$ iff $N\mid M$. For $N\in\mathcal{N}$, consider the circle $X_{N}:=\R/N\Z\cong \T$, written additively. For $N\mid M$ there is a covering map
\begin{equation}\label{eq:cover}
\pi_{M\to N}: X_{M}\longrightarrow X_{N},\qquad [t]_{M}\longmapsto \Big[\tfrac{M}{N}t\Big]_{N},
\end{equation}
of degree $M/N$. The inverse system $\{X_{N},\pi_{M\to N}\}$ has inverse limit
\begin{equation}\label{eq:solenoid}
\varprojlim_{N\in\mathcal{N}} X_{N}\ =\ \Big\{(x_{N})_{N}\in \prod_{N} X_{N}\ :\ \pi_{M\to N}(x_{M})=x_{N}\ \text{ whenever } N\mid M\Big\},
\end{equation}
a compact abelian group known as the (universal) one-dimensional solenoid.

\begin{proposition}\label{prop:dual-as-solenoid}
There is a canonical topological group isomorphism
\[
\what{\Q}\ \cong\ \varprojlim_{N\in\mathcal{N}} X_{N}.
\]
\end{proposition}

\begin{proof}
For each $N$, the subgroup $N^{-1}\Z/\Z\subset \Q/\Z$ is cyclic of order $N$, and $\Hom(N^{-1}\Z/\Z,\T)\cong \Z/N\Z$ identifies with $X_{N}$ by sending $1/N+\Z$ to $e^{2\pi i t/N}$ with $t\in \R/N\Z$. As $\Q/\Z=\varinjlim_{N} (N^{-1}\Z/\Z)$, Pontryagin duality exchanges direct limits with inverse limits for discrete abelian groups, yielding
\(
\what{\Q/\Z}\cong \varprojlim_{N} X_{N}.
\)
Finally, since $\Q$ is divisible, the exact sequence $0\to \Z\to \Q\to \Q/\Z\to 0$ dualizes to $0\to \what{\Q/\Z}\to \what{\Q}\to \what{\Z}\to 0$ with $\what{\Z}\cong \T$ and the last map trivial; hence $\what{\Q}\cong \what{\Q/\Z}$ canonically. Thus $\what{\Q}\cong \varprojlim_{N} X_{N}$.
\end{proof}

By Gelfand duality, Proposition~\ref{prop:dual-as-solenoid} yields
\begin{equation}\label{eq:direct-limit-algebra}
C(\what{\Q})\ \cong\ \varinjlim_{N\in\mathcal{N}}\ C(X_{N}),
\end{equation}
where for $N\mid M$ the connecting $*$-homomorphism
\begin{equation}\label{eq:connecting-map}
\phi_{N,M}: C(X_{N})\longrightarrow C(X_{M}),\qquad \phi_{N,M}(f)=f\circ \pi_{M\to N},
\end{equation}
is the pullback along the degree-$M/N$ covering \eqref{eq:cover}. Since $X_{N}\cong \R/N\Z$, one may equally write
\(
C(X_{N})\cong C(\R/N\Z),
\)
which gives the desired description
\begin{equation}\label{eq:cstarQ-as-limit}
\cstar(\Q)\ \cong\ C(\what{\Q})\ \cong\ \varinjlim_{N} C(\R/N\Z).
\end{equation}

Since $C^*(\mathbb{Q})$ itself can be realized as an inductive (direct) limit of $C^*$-algebras, 
the continuity of $K$-theory with respect to inductive limits \cite[Ch.~6]{rordam} 
implies that
\[
K_{*}\big(\cstar(\Q)\big)\ \cong\ \varinjlim_{N}\, K_{*}\big(C(X_{N})\big).
\]

For each $N$, $X_{N}\cong \T$, so
\[
K_{0}\big(C(X_{N})\big)\cong \Z\cdot[1],\qquad
K_{1}\big(C(X_{N})\big)\cong \Z\cdot[u_N],
\]
where $[1]$ is the class of the unit and $[u_N]$ is the class of the canonical unitary $u_N(z)=z$ on $\T$ (under a fixed identification $X_{N}\cong \T$).

For $N \mid M$, the connecting map $\phi_{N,M}$ induces the identity on $K_0$, 
while on $K_1$ it corresponds to multiplication by $M/N$ under the standard 
identifications $K_0(C(X_N))\cong \Z$ and $K_1(C(X_N))\cong \Z$. 
In other words, $(\phi_{N,M})_*([1])=[1]$ and $(\phi_{N,M})_*([u_N])=(M/N)\,[u_M]$.

Combining the description of the connecting maps on $K$-theory with the continuity of $K$-theory under inductive limits, 
we can now state the final computation of the $K$-groups of $C^*(\mathbb{Q})$ in a concise form:

\begin{theorem}\label{thm:main}
For the additive discrete group $\Q$,
\begin{align*}
K_{0}\big(\cstar(\Q)\big)
 &\cong \varinjlim_{N}\ (\Z,\ \id)\ \cong\ \Z,\\[0.3em]
K_{1}\big(\cstar(\Q)\big)
 &\cong \varinjlim_{N}\ \big(\Z,\ \times\tfrac{M}{N}\big)\cong \Q.
\end{align*}
Under the identification $\cstar(\Q)\cong \varinjlim_{N} C(\R/N\Z)$, the generator of $K_{0}$ is the class $[1]$, 
and the isomorphism $K_{1}\cong \Q$ arises naturally from the corresponding direct-limit construction.

\end{theorem}

In summary, by exploiting the continuity of $K$-theory with respect to inductive limits, 
we have obtained the $K$-groups of $C^*(\mathbb{Q})$ in a concise and conceptually straightforward manner. 
Having established this independent argument, we now turn to a direct, elementary computation, 
which establishes a connection between the topological properties of elements in $C^*(\mathbb{Q})$ and their equivalence classes in $K$-theory, 
and can be readily adapted to compute K-theory of the $C^*$-algebra of the adele ring $\A$.

\section{Introduction to pre-periodic functions}

In this section,  we introduce the concept of pre-periodic functions. After proving the existence of the lifts of pre-periodic functions, we define a homotopy invariant termed the generalized winding number. This invariant plays a crucial role in classifying homotopy equivalence classes, laying the groundwork for K-theoretic computations.

\begin{definition}
    Let \( f \in C(\mathbb{R}, S^1) \) and suppose that it satisfies the following properties. For any \( \varepsilon > 0 \), there exists \( N\in \mathbb{Z}^+ \) such that for all \( k \in \mathbb{Z} \) and \( x \in \mathbb{R} \), the following condition holds:
\begin{equation}
    \lvert f(x + kN) - f(x) \rvert < \varepsilon.
\end{equation}
Such functions are called pre-periodic functions.

\end{definition}
Functions with rational periods are naturally pre-periodic functions.The following example is used to illustrate the unique characteristics of pre-periodic functions compared to periodic functions.

\begin{example}
Take a function on \(\mathbb{R}\) in a recursive form, given by:
\[ 
f(x) = \begin{cases} 
\frac{1}{3} x, \\\hspace{21em}x\in[-1,1), \\
\frac{3^n + \frac{1}{3^{n+2}}}{3^n - \frac{1}{3^{n+1}}} \left( f(x + 2 \times 3^n) + \frac{1}{2} \left( 3^n - \frac{1}{3^{n+1}} \right) - \frac{1}{2} \left( 3^{n+1} - \frac{1}{3^{n+2}} \right) \right), \\\hspace{21em}x\in [-3^{n+1},-3^n), \\[2ex] 
\frac{3^n + \frac{1}{3^{n+2}}}{3^n - \frac{1}{3^{n+1}}} \left( f(x - 2 \times 3^n) + \frac{1}{2} \left( 3^n - \frac{1}{3^{n+1}} \right) \right) + \frac{1}{2} \left( 3^n - \frac{1}{3^{n+1}} \right), \\\hspace{21em}x\in [3^{n},3^{n+1}).
\end{cases}
\]

For any $n\in \mathbb{N},x\in [-3^n,3^n)$, it can be obtained through calculation that:
\begin{align*}
    \left| f(x + 2 \times 3^n) - f(x)-3^n \right| < \frac{1}{3^{n+1}},\\
    \left| f(x) - f(x- 2 \times 3^n)-3^n \right| < \frac{1}{3^{n+1}}.
\end{align*}
Furthermore, we have:
\[
\left| f(x + 2 \times 3^nk) - f(x) -  3^nk \right| <  \frac{1}{2}\times\frac{1}{3^n} , \forall n\in \mathbb{N}, k\in \mathbb{Z}, x\in [-3^n,3^n).
\]
Thus, $e^{2\pi i f(x)}$ is a pre-periodic function. By mathematical induction, we can prove that all points where $f$ is differentiable with derivative $\frac{1}{2}$ belong to the interval $(-1, 1)$, indicating that $e^{2\pi i f(x)}$ is not a periodic function.

Figure \ref{fig:sample_image_a} intuitively shows the ``approximate periodicity" of pre-periodic function. In fact, the function on \([3, 9],[-9,3]\) is obtained by vertically stretching the function on \([-3, 3]\) by a factor of $\frac{3^4+1}{3^4-3}$ and then splicing them together.

\begin{figure}[h!]
    \centering
    \includegraphics[width=\textwidth]{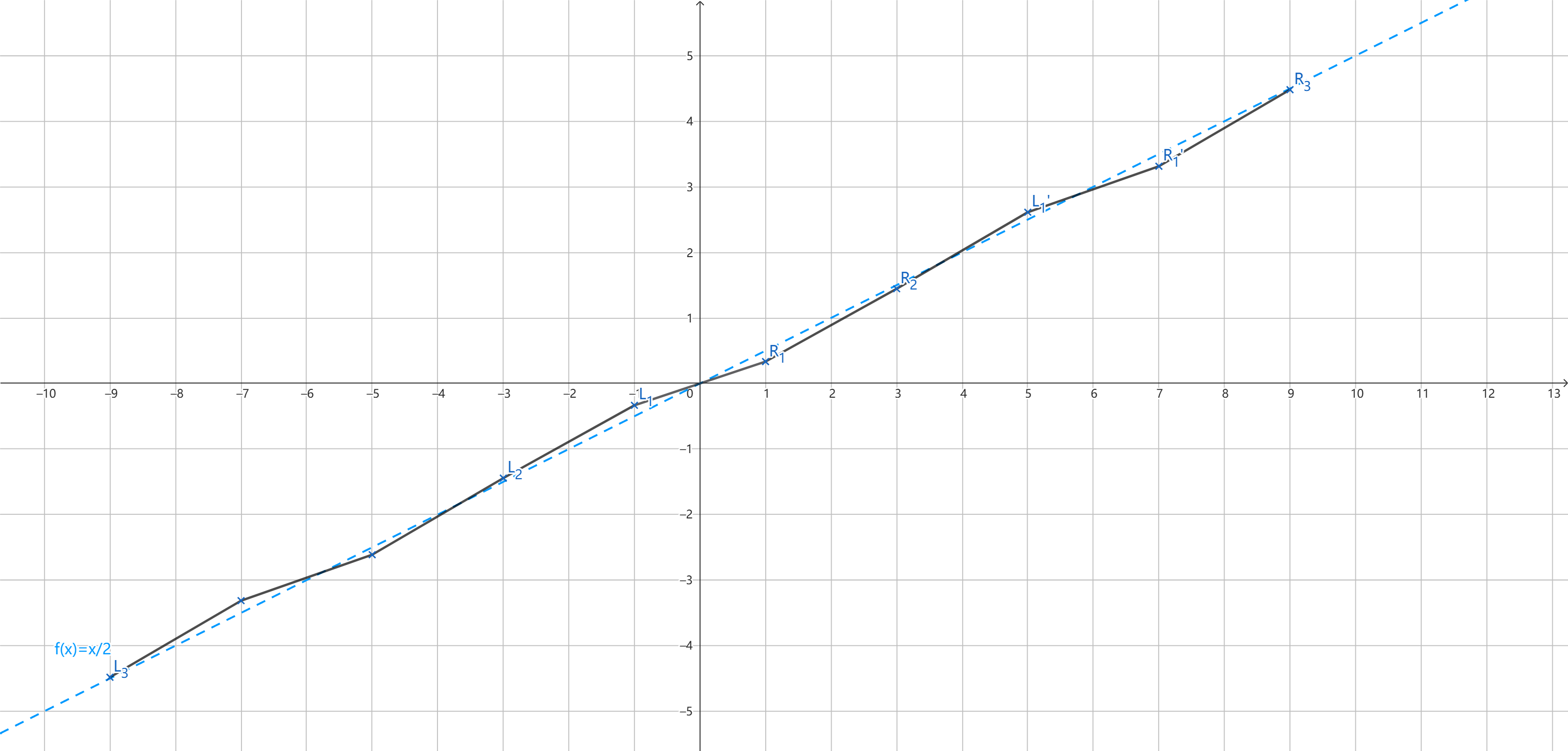}
    \caption{The graph of \( f(x) \) within the interval \([-9, 9]\).
}
    \label{fig:sample_image_a}
\end{figure}
\FloatBarrier
\end{example}

The following lemma is immediate.
\begin{lemma}\label{unc}
A pre-periodic function is uniformly continuous.
\end{lemma}
This property enables these functions to be  lifted to \(C(\mathbb{R,\mathbb{R}})\).
\begin{lemma}\label{lift}
    Let the map \( p: \mathbb{R} \to S^1 \) be defined by \( p(x) = e^{2\pi ix} \). For any pre-periodic function \( f \in C(\mathbb{R}, S^1) \), and for any \( e \in p^{-1}(f(r)) \), there exists a function \( \tilde{f} \in C(\mathbb{R}, \mathbb{R}) \) such that the following conditions hold:
\begin{enumerate}
    \item \( \tilde{f}(r) = e \),
    \item \( p \circ \tilde{f} = f \).
\end{enumerate}
\end{lemma}
\begin{proof}
    Without loss of generality, we may assume \( r = 0 \). By Lemma \ref{unc}, take \( \varepsilon = 1 \). There exists \( m \in \mathbb{N} \) such that for all \( |x - y| \leq \frac{1}{m} \), we have \( |f(x) - f(y)| < 1 \).

Divide \( \mathbb{R} \) into countably many intervals \( \cdots, I_{-1}, I_0, I_1, \cdots \) where \( I_i = \left[\frac{i}{m}, \frac{i+1}{m}\right] \).

If \( f \) restricted to \( I_i \) is not surjective, we can successively define the lifts \( \tilde{f}_i \) such that:
\begin{itemize}
    \item \( \tilde{f}_0(0) = e \),
    \item \( \tilde{f}_i\left(\frac{i}{m}\right) = \tilde{f}_{i-1}\left(\frac{i}{m}\right) \) for all \( i \in \mathbb{Z} \).
\end{itemize}

Combine each \( \tilde{f}_i \) to form the lift \( \tilde{f} \) of \( f \).

\textbf{Uniqueness:} If there exists \( \tilde{f}' \) such that
\begin{enumerate}
    \item \( \tilde{f}'(0) = e \),
    \item \( p \circ \tilde{f}' = f \),
\end{enumerate}
then \( p(\tilde{f}' - \tilde{f}) \equiv 1\), \( \tilde{f}' - \tilde{f} \)\ is identically zero.
\end{proof}
The following lemma establishes the continuity of lifting: For any two functions that are sufficiently close in their original function space, there exist corresponding lifts in the target space that preserve this proximity. 
\begin{lemma}\label{lift8}
Let \( \tilde{f} \) be a lift of the pre-periodic function \( f \). Then, for \( \varepsilon \in [0, \frac{1}{8}) \), if the pre-periodic function \( g \) satisfies \( |g - f| < \varepsilon \), then there exists a unique lift \( \tilde{g} \) of \( g \) satisfies \( |\tilde{g} - \tilde{f}| < \frac{\varepsilon}{4}\).
\end{lemma}

\begin{proof}
    Let the principal argument function on \( S^1 \) be \( \text{arg}: S^1 \to (-\pi, \pi] \). For any \(x\in S^1\) such that the distance between $x$ and $1$ is less than 1, we have
\begin{align*}
\left| \arg(x) \right| &= \left| \arg(x) - \arg(1) \right| \\
&= 2 \arcsin\left( \frac{|x - 1|}{2} \right) \\
&< \frac{\pi}{2} |x - 1|
\end{align*}
Let $t_0=\frac{1}{2\pi} \arg \left(g(0) \cdot f(0)^{-1} \right)+\tilde{f}$, $p(t_0)=e^{2\pi i t_0}=g(0)$, and

\begin{align*}
\left| t_0-\tilde{f} \right| &= \left| \frac{1}{2\pi} \arg \left(g(0) \cdot f(0)^{-1} \right) \right| \\
&= 2 \arcsin\left( \frac{|g(0) \cdot f(0)^{-1} - 1|}{2} \right) \\
&< \frac{1}{4} |g(0) \cdot f(0)^{-1} - 1|\\
&< \frac{\varepsilon}{4}.
\end{align*}

Thus, there exists a unique lift \( \tilde{g} \) such that \( \tilde{g}(0) = t_0 \).\\
Now we prove by contradiction. Suppose that there exists $x \in \mathbb{R}$ such that $|\tilde{g}(x) - \tilde{f}(x)| > \frac{\varepsilon}{4}$. Let

$$
t_1 = \tilde{f}(x) + \frac{1}{2\pi} \arg (g(x) \cdot f(x)^{-1}),
$$
and
$$
p(t_1) = p(\tilde{g}(x))=g(x), \quad t_1 \neq \tilde{g}(x),
$$
so
$$
|\tilde{g}(x) - \tilde{f}(x)| > 1 - \varepsilon > \frac{7}{8}.
$$
We also have
$$
|\tilde{g}(0) - \tilde{f}(0)|<\frac{\varepsilon}{4}.
$$
Thus, there exists $x_0 \in \mathbb{R}$ such that
$$
|\tilde{g}(x_0) - \tilde{f}(x_0)|=\frac{1}{2}\Rightarrow |f(x_0)-g(x_0)|=1.
$$
This leads to a contradiction.
\end{proof}
With the lifting of pre-periodic functions, we can define the following generalized winding number as a homotopy invariant. 
\begin{theorem}\label{gwd}
f is a pre-periodic function, then the following limit exists and is a rational number:
$$
\lim_{x\to+\infty}\frac{\tilde{f}(x)-\tilde{f}(-x)}{2x}.
$$
This limit is called the generalized winding number of \(f\), denoted by \(\#f\).
\end{theorem}
\begin{proof}
    Because \(f\) is a pre-periodic function, for \(\varepsilon =\frac{1}{8}\) we have $N_{\varepsilon}\in\mathbb{N}$ such that the following equation holds:
    $$
    |f(x+kN_\varepsilon)-f(x)|<\frac{1}{8},\qquad \forall k\in \mathbb{Z}, x\in \mathbb{R}.
    $$
Let \( f^k(x) = f(x + kN_{\varepsilon}) \), and define \( f^0(x) = f(x) \). Consider the lift \( \tilde{f}(x) \) of \( f(x) \), by Lemma \ref{lift8}, there exists a lift \( \widetilde{f^k} \) of \( f^k \) such that
\[
|\widetilde{f^k}(x) - \tilde{f}(x)| < \frac{1}{32}.
\]
Note that \( p(\tilde{f}(x + kN_{\varepsilon})) = f^k(x) \), \(\tilde{f}(x + kN_{\varepsilon})\) and \(\widetilde{f^k}(x)\) differ by an integer. Therefore, there exists an integer \( n_k \) such that
\[
|\tilde{f}(x + kN_{\varepsilon}) - n_k - \tilde{f}(x)| < \frac{1}{32}.
\]

Using mathematical induction, we would like to prove that \( n_k = kn_1 \), that is
\begin{enumerate}
    \item \( n_1 = 1 \cdot n_1 \), this is trivially satisfied,
    \item Suppose \( n_k = kn_1 \). 
\end{enumerate}
Consider the following three equations:
\begin{align*}
        |\tilde{f}(x + kN_{\varepsilon}&) - \tilde{f}(x) - n_k| < \frac{1}{32}. \\
    |\tilde{f}(x + (k+1)N_\varepsilon) &- \tilde{f}(x + kN_\varepsilon) - n_1| < \frac{1}{32},\\
|\tilde{f}(x + (k+1)&N_\varepsilon) - \tilde{f}(x) - n_{k+1}| < \frac{1}{32}.
 \end{align*}
 So \(|n_{k+1}-n_k-n_1|<\frac{3}{32}\) implies \( n_{k+1} = n_k + n_1 = (k+1)n_1 \). Finally,
\[
\lim_{x \to +\infty} \frac{\tilde{f}(x) - \tilde{f}(-x)}{2x} = \lim_{x \to +\infty} \frac{\tilde{f}\left(\lfloor \frac{x}{N_\varepsilon}\rfloor \right) - \tilde{f}\left(-\lfloor \frac{x}{N_\varepsilon}\rfloor \right)}{2x} = \lim_{x \to +\infty} \frac{2\lfloor \frac{x}{N_\varepsilon}\rfloor \cdot n_1}{2x} = \frac{n_1}{N_\varepsilon}.
\]
This limit does not depend on the choice of the lift.
\end{proof}
The generalized winding number possesses properties analogous to those of the classical winding number. 
\begin{proposition}\label{prop}
    The generalized winding number has the following properties:
\begin{enumerate}
    \item \( \#(fg) = \#f + \#g \)
    \item  The generalized winding number is locally constant.
\end{enumerate}
\end{proposition}
\begin{proof}
    1.\( \tilde{f} + \tilde{g} \) is a lift of \( f \cdot  g \). Then, we have
    $$
    \#(fg)=\lim_{x\to+\infty}\frac{\widetilde{fg}(x)-\widetilde{fg}(-x)}{2x}=\lim_{x\to+\infty}\frac{\tilde{f}(x)+\tilde{g}(x)-\tilde{f}(-x)-\tilde{g}(-x)}{2x}=\#f+\#g.
    $$
    2.  If \( \sup\limits_{x \in \mathbb{R}} |f(x) - g(x)| < \frac{1}{8} \), then by Lemma \ref{lift8} there exist lifts of \( f \) and \( g \) that satisfy the following conditions:
    $$
     \sup\limits_{x \in \mathbb{R}} |\tilde{f}(x) - \tilde{g}(x)| < \frac{1}{32} .
    $$
So $\lim_{x\to+\infty}\frac{|\tilde{f}(x)-\tilde{f}(-x)-\tilde{f}(x)+\tilde{g}(-x)|}{2x}<\lim_{x\to+\infty}\frac{1/16}{2x}=0.$
\end{proof}
The concept of generalized winding numbers will enable us to construct homotopy invariants for adelic periodic functions from local to global in the next section. 
\begin{remark}\label{rem}
In fact, for functions \( f \in C(\mathbb{R}, S^1) \) that share identical limits at both \( +\infty \) and \( -\infty \), an analogous analysis applies: such functions similarly admit a unique lift \( \tilde{f} \in C(\mathbb{R}, \mathbb{R}) \), allowing the definition of a winding number \( \#f \):
    \begin{equation}
        \#f=\lim_{x\to+\infty}\tilde{f}(x)-\tilde{f}(-x),
    \end{equation}
which inherits the properties articulated in Proposition 1.5. This concept is essential when studying the $K$-theory of \(C^*(\mathbb{A})\).

\end{remark}
\section{Periodic adelic function and pre-periodic function}
The goal of this section is to describe the structure of \( K_1(C^*(\mathbb{Q})) \) by identifying all homotopy equivalence classes of unitary elements in \( C_0(\mathbb{A}/\mathbb{Q}) \). We introduce a new definition: the generalized winding number, and prove that the existence of a homotopy between two elements is determined by whether this quantity is equal. The lifting of periodic adelic functions provides an effective approach to the construction of homotopy, and the following  disjoint partition of \( \mathbb{A}_f \) is the first step in this procedure. The properties of p-adic numbers can be found in the reference \cite{MR2807433}
\begin{lemma}\label{china}
    Let \( A_f \) denote \(\prod\limits_{p \text{ prime}} \mathbb{Q}_p
 \), which is the restricted product of all \( \mathbb{Q}_p \),  \( e_p(N) \) denote the exponent of the prime \( p \) in a positive integer N's prime factorization, i.e. 
\(
N = \prod\limits_{p \text{ prime}} p^{e_p(N)},
\)
and \(K_N:=\prod\limits_{p \text{ prime}}p^{e_p(N)}\mathbb{Z}_p\). 
Then there exists the following decomposition:

$$
\A_f=\bigsqcup_{\alpha \in \mathbb{Q}/N\mathbb{Z}} \alpha+K_N.
$$
\end{lemma}
\begin{proof}
Our proof strategy is as follows: For any \( x \in \mathbb{A}_f \), we aim to find a rational number \( \alpha \) such that \( x - \alpha \in K_N \). Additionally, we must show that if two rational numbers \( \alpha_1 \) and \( \alpha_2 \) both satisfy this condition, then 
\[
\alpha_1 - \alpha_2 \equiv 0 \pmod{N},
\]
and consequently 
\[
\alpha_1 + K_N = \alpha_2 + K_N.
\]
This ensures the disjointness of the union in construction within the lemma.

First we choose an $x_f=\{x_2,x_3,x_5\dots\}\in \A_f$, and construct a finite set of prime numbers:
\[
S:=\{p \text{ is prime}| x_p\notin p^{e_p(N)}\mathbb{Z}_p\}.
\]
Because \(x_p\) belongs to \(\mathbb{Z}_p\) and \(e_p(N)=0\) holds for almost every prime number \(p\), $S$ is finite. Denote the elements in $S$ by \(p'_1, p'_{2},\dots,p'_{\lvert S\rvert}\). 

Applying Chinese Remainder Theorem, for any $p'_{i}\in S$, we can find an integer \(M_i\) such that:
\begin{align}
    1, M_i&\equiv 1\ (\textbf{mod} (p'_i)^{e_{p'_i}(N)}).\\
    2, M_i&\equiv 0\ (\textbf{mod}\ (p^{e_p(N)}),\ \forall p\neq p'_i.
\end{align}

On the other hand, take \(p'_i \in S\), there exists an $\alpha_i\in \mathbb{Q}$ satisfying
\[
\alpha_i+(p'_i)^{e_{p'_i}(N)}\mathbb{Z}_{p'_i}=x_{p'_i}+(p'_i)^{e_{p'_i}(N)}\mathbb{Z}_{p'_i},
\]
and the denominator of \(\alpha_i\) is a power of \(p'_i\). It follows that for any $p\neq p'_i$, \(\alpha_i\) is $p$-adic integer.

As a result, we obtain the rational number:
\[
\alpha:=\sum\limits_{i=1}^{\lvert S \rvert}\alpha_i\cdot M_i.
\]
The current task involves verifying \( x - \alpha \in K_N \). For a fixed $p'_j\in S$, we decompose \(\alpha\) into the sum of two parts:
\[\alpha_{j}\cdot M_{j}+\sum\limits_{i\neq j}\alpha_i\cdot M_i.\]
 \(M_{j}\equiv 1\ (\textbf{mod} (p'_j)^{e_{p'_j}(N)})\) implies that \(x_{p'_j}-\alpha_{j}\cdot M_{j}\in(p'_j)^{e_{p'_j}(N)}\mathbb{Z}_{p'_j}\). Since the subsequent summation terms belong to \( (p'_j)^{e_{p'_j}(N)}\mathbb{Z}_{p'_j} \), we have
\[
x_{p'_j} - \alpha \in (p'_j)^{e_{p'_j}(N)} \mathbb{Z}_{p'_j}.
\]
 For $p\notin S$, $\alpha_i$ is a $p-$adic integer for all $i$, while each \( M_i \) is divisible by \( p^{e_p(N)} \), so both $\alpha$ and $x_p-\alpha$ belong to $p^{e_p(N)}\mathbb{Z}_p$.

Given $x_f\in \A_f$, we have found $\alpha\in \mathbb{Q}$ such that $x_f\in \alpha+K_N$. Let $\alpha'\in \mathbb{Q}$ satisfy the same conditions. Then this leads to
$$
x_f-\alpha\in K_N,\
x_f-\alpha'\in K_N\Rightarrow
\alpha-\alpha'\in K_N.
$$
We obtain that $\alpha-\alpha'\equiv0\ \text{mod}\ N$, then $\alpha+K_N=\alpha'+K_N.$
\end{proof}
By utilizing  the property of adelic periodic functions that they can have any rational number as a period, we can extend the concept of generalized winding numbers to adelic periodic functions. 
\begin{theorem}\label{gonggongwind}
    Let \( g \in C(\mathbb{A}/\mathbb{Q}, S^1) \), \( x_f \in \mathbb{A}_f \), then the associated element in \( C(\mathbb{R}, S^1) \) given by  
\[
g_{x_f}(x_\infty) := g(\{x_\infty, x_f\})
\]  
is a pre-periodic function. Futhermore, for any \( x_f \in \mathbb{A}_f\), we have \(\#g_{x_f}=\#g_{0_f}\), where \(0_f\) is the embedding of 0 in \(\A_f\). Therefore, we can define the generalized winding number of \( g \in C(\mathbb{A}/\mathbb{Q}, S^1) \) by
$$
\#g:=\#g_{0_f}.
$$
\end{theorem}
\begin{proof}
We define a metric on \(\A_f\) by
\[
|x_f|_{\A_f} = 
\begin{cases} 
\max \left( \frac{|x_p|_p}{p} \right), & \forall p, x_p \in \mathbb{Z}_p \\
\max |x_p|, & \text{otherwise}.
\end{cases}
\]
Thus we can define a metric on $\mathbb{A}$ by
$$
|x|_{\A}=|\{x_{\infty,},x_f\}|_{\A}=|x_{\infty}|_{\infty}+|x_f|_{\A_f}.
$$
For the equivalence between the topology induced by this metric and the restricted product topology, see Appendix \ref{metric}.

 Since \( \mathbb{A}/\mathbb{Q} \) is a compact \( C_1 \)-metric space, and any continuous functions $g$ on it are uniformly continuous.

For \( \varepsilon > 0 \), there exists \( \delta > 0 \) such that

\[
\forall x, y \in \mathbb{A}, \quad |x - y|_\mathbb{A} < \delta \quad \Rightarrow \quad |g(x) - g(y)| < \varepsilon.
\]

Let \( N \) be a positive integer such that for any prime $p$, \(
 \quad \frac{|N|_p}{p} < \delta.
\)
Then for any \( x = \{ x_\infty, x_2, x_3, x_5, x_7,\dots \} = \{ x_\infty, x_f \} \in \A,\ 
k \in \mathbb{Z}\), we have \( |x - (x_\infty, x_f + kN)| < \delta.
\)
Therefore,
\begin{align*}
    \varepsilon &> |g(x) - g(\{ x_\infty, x_f + kN \})| \\&= |g(x) - g(\{ x_\infty - kN, x_f \})| \\&= |g_{x_f}(x_\infty) - g_{x_f}(x_\infty - kN)|,\forall k\in \mathbb{Z}, x_{\infty}\in \mathbb{R}.
\end{align*}
This shows that $g_{x_f}$ is a pre-periodic function. 

With Lemma \ref{prop}, we can find $N\in \mathbb{N}$ such that for any $x_f\in K_N,\ \#g_{x_f}=\#g_{0_f}$. Then for any \(x_f\in \A_f \), there exists an $\alpha \in \mathbb{Q}/N\mathbb{Z}$ satisfing \( x_f-\alpha \in K_N\) by Lemma \ref{china}, which follows that
$$
\#g_{0_f}(x_{\infty})=\#g_{x_f-\alpha}(x_{\infty})=\#g_{x_f}(x_{\infty}+\alpha)=\#g_{x_f}.\qedhere
$$
\end{proof}
We proceed to provide an example of computing the generalized winding number for adelic periodic functions.
Define an additive adelic character $e:\A\rightarrow \mathbb{C}$. For \(x=\{x_\infty,x_2,x_3,\dots\}\in \A\), let
\[
e(x)=\prod\limits_{v\leq\infty}e_v(x_v)
\]
with \(e_p(x_p)=e^{-2\pi i\{x_p\}}\) if \(p\leq \infty\) and \(e_{\infty}(x_{\infty})=e^{2\pi ix_{\infty}}\) if \(v=\infty\). Then $e_{0_f}(x_\infty)=e^{-2\pi i\cdot x_{\infty}}$, $\#e=\#e_{0_f}=\lim_{x\rightarrow \infty}\frac{- x-(x)}{2x}=-1.$

To establish relationships between functions with the same generalized winding number using simple linear homotopy, it is necessary to lift functions from \( C(\A/\mathbb{Q}, S^1) \) to \( C(\A/\mathbb{Q},\mathbb{R}) \).
\begin{lemma}\label{adelelift}
     For \( g \in C(\A/\mathbb{Q}, S^1) \) with \( \#g = 0 \), there exists \( \tilde{g} \in C(\A/\mathbb{Q}, \mathbb{R}) \) such that \( g = p \circ \tilde{g} \), where \(p(x)=e^{2\pi ix}\in C(\mathbb{R}).\)
\end{lemma}
\begin{proof}
    Let \( g_{0_f}(x_\infty) = g(\{x_{\infty}, 0_f\}) \) , which is a pre-periodic function on \( \mathbb{R} \). Following Lemma \ref{lift}, there exists \( \widetilde{g_{0_f}} \in C(\mathbb{R}, \mathbb{R}) \) such that \( p(\widetilde{g_{0_f}}) = g_{0_f} \).

Due to uniform continuity of $g$, we can find a $\delta > 0$ such that 
\[
|x - y|_{\A} < \delta \quad \Rightarrow \quad |g(x) - g(y)| < \frac{1}{8}.
\]
 Next, we select a suitable positive integer \( N \) such that the diameter of the set $K_N$ is less than $\delta$, which implies that
\[
\lvert g_{x_f}(x_\infty) - g_{0_f}(x_\infty)\rvert= \lvert g(\{x_{\infty}, x_f\}) - g(\{x_{\infty}, 0_f\})\rvert < \frac{1}{8},\ \forall x_f\in K_N.
\]

According to Lemma \ref{lift8}, there exists a unique lift \(\widetilde{g_{x_f}}\) of \( g_{x_f} \), satisfying:
\[
|\widetilde{g_{x_f}}(x_\infty) - \widetilde{g_{0_f}}(x_\infty)| < \frac{1}{32}.
\]

As stated in Lemma \ref{china},
\[
\A_f = \bigsqcup_{\alpha\in \mathbb{Q}/N\mathbb{Z}} \left( \alpha + K_N\right),
\]
then it follows that for any \( x_f \in \A_f \), there exists \( \alpha \in \mathbb{Q}/N\mathbb{Z} \) such that \( x_f + \alpha \in K_N \). 

For \( g \in C(\mathbb{A}/\mathbb{Q},S^1) \), we have:
\[
g_{x_f}(x_\infty) = g_{x_f + \alpha}(x_\infty+\alpha).
\]

Thus, by Lemma \ref{lift}, the function \( g_{x_f} \) has a unique lift at \( 0_{\infty} \) with value \( \widetilde{g_{x_f + \alpha}}(0_\infty+\alpha) \), which is denoted as \( \widetilde{g_{x_f}} \).

Define the map \( \widetilde{g}: \{x_{\infty}, x_f\} \mapsto \widetilde{g_{x_f}}(x_\infty) \) from \( \A \) to \( \mathbb{R} \). It is evident that $g=e^{2\pi i \tilde{g}}.$ Now we need to show that $\tilde{g}$ is continuous and $\mathbb{Q}$-periodic. To prove $\tilde{g}$ is continuous, we need to consider its continuity on $\mathbb{R}$ and $\A_f$. Since each $\widetilde{g_{x_f}}$ is the lift of a pre-periodic function (and hence a uniformly continuous function), $\tilde{g}$ is uniformly continuous with respect to the variable $x_\infty$. By the proof of Theorem \ref{gonggongwind}, $g$ is uniformly continuous, then for any $0<\varepsilon<\frac{1}{8}$, we can choose a $\delta >0$ such that:
\begin{align}
    1&, \lvert x_f-y_f\rvert_{\A_f}<\delta \Rightarrow x_f-y_f\in K_N, \forall x_f,\ y_f\in \mathbb{A}_f.\\
    2&,\lvert x_f-y_f\rvert_{\A_f}<\delta \Rightarrow \sup\lvert g_{x_f}-g_{y_f}\rvert<\varepsilon.
\end{align}
For the $x_f,y_f$ above, we have
\[
\lvert \widetilde{g_{x_f}}(x_\infty)-\widetilde{g_{y_f}}(x_\infty)\rvert<\frac{\varepsilon}{4},
\]
according to Lemma \ref{lift8}. Therefore, $\tilde{g}$ is uniformly continuous for both $x_\infty$ and $x_f$.

Finally, we prove that \( \widetilde{g} \in C(\A / \mathbb{Q}, \mathbb{R}) \), i.e. for any \( x_f \in \A_f, \alpha\in\mathbb{Q} \), it holds that:
\[
\widetilde{g}(\{x_\infty, x_f\}) = \widetilde{g}(\{x_\infty + \alpha, x_f + \alpha\}).
\]

With the positive integer $N$ given above, we divide the rational number $\alpha$ into the sum of $\alpha_1$ and $\alpha_2$, where \( \alpha_1 \in N\mathbb{Z} \) and \( \alpha_2 \in \mathbb{Q} /N\mathbb{Z} \). Let us further assume that $x_f\in \beta +K_N$ for some $\beta\in \mathbb{Q}/N\mathbb{Z}$.

Part 1: For the integer $N$ given above, we have
\[
|g_{x_f}(x_\infty) - g_{x_f}(x_\infty + kN)| < \frac{1}{8}, \ \forall x_\infty \in \mathbb{R}, k \in \mathbb{Z}.
\]
In view of Theorem \ref{gwd}, it is possible to find an integer $n_1$ such that:
\[
|\widetilde{g_{x_f}}(x_\infty) - \widetilde{g_{x_f}}(x_\infty + kN)-k\cdot n_1| \leq \frac{1}{32}.
\]
However, the generalized winding number of \(g_{x_f}\) is zero, which equals to $\frac{n_1}{N}.$ It follows that 
\[
|\widetilde{g_{x_f}}(x_\infty) - \widetilde{g_{x_f}}(x_\infty + kN)| \leq \frac{1}{32}.
\]
Additionally, 
\[
g_{x_f + kN}(x_\infty) = g(\{x_\infty, x_f + kN\}) = g(\{ x_\infty - kN, x_f \})=g_{x_f}(x_{\infty}-kN).
\]
So \( \widetilde{g_{x_f+kN}}(x_\infty ) \) and \( \widetilde{g_{x_f}}(x_\infty-kN) \)  differ by an integer.

Note that \(x_f\) and\( x_f+kN\) both belong to $\beta +K_N$, so \[|\widetilde{g_{x_{f}}}(x_{\infty})-\widetilde{g_{x_{f}+kN}}(x_{\infty})|<\frac{1}{32},\]while 
\[
|\widetilde{g_{x_{f}}}(x_{\infty})-\widetilde{g_{x_{f}}}(x_{\infty}+kN)|<\frac{1}{32}.
\]
Hence, for every \(\ k\in \mathbb{Z},x_{\infty}\in \mathbb{R},x_{f}\in \beta+K_N\)
\begin{align*}
    |\widetilde{g_{x_{f}+kN}}(x_{\infty})&-\widetilde{{g}_{x_{f}}}(x_{\infty}-kN)|<\frac{1}{16}.
\end{align*}
Thus,
\begin{align*}
    \widetilde{g_{x_{f}+kN}}&(x_{\infty})=\widetilde{{g}_{x_{f}}}(x_{\infty}-kN).
\end{align*}
We can replace \( kN \) with \( \alpha_1 \) to complete the first part of the proof.

Part 2: From the definition of \( \tilde{g} \),  for  \( x_f \in\beta+ K_N \), we have\[\widetilde{g_{x_f}}(0_{\infty})=\widetilde{g_{x_f-\beta}}(-\beta).\]
However, \(\widetilde{g_{x_f}}(x_{\infty})-\widetilde{g_{x_f-\beta}}(x_{\infty}-\beta)\) is a lift of the function
\[g_{x_f}(x_{\infty})\cdot g_{x_f-\beta}(x_{\infty}-\beta)^{-1},\]
which is equal to 1 for all $x_{\infty}\in \mathbb{R}$. Thus,
\[\tilde{g}(\{x_\infty, x_f\})=\widetilde{g_{x_f}}(x_{\infty})=\widetilde{g_{x_f-\beta}}(x_{\infty}-\beta)=\tilde{g}(\{x_\infty-\beta, x_f-\beta\}).\]
Similarly, for $x_f-\beta\in K_N, \beta+\alpha_2\in \mathbb{Q}$, it can be proved that:
\[
\widetilde{g_{x_f-\beta}}(x_{\infty}-\beta)=\widetilde{g_{x_f+\alpha_2}}(x_{\infty}+\alpha_2).
\]
We arrive at the final conclusion:
\[
\tilde{g}(\{x_\infty, x_f\})=\tilde{g}(\{x_\infty+\alpha, x_f+\alpha\}), \forall \{x_\infty,x_f\}\in\A, \alpha\in \mathbb{Q}.
\]
\end{proof}

    Now let \( f, g \in C(\A/\mathbb{Q}, S^1) \), and suppose that \( \#f = \#g = \alpha \in \mathbb{Q} \). Note that the generalized winding numbers of \( e(\alpha x)f \) and \( e(\alpha x)g  \) are both 0. By Lemma \ref{adelelift}, their lifts exist, denoted by 
\[
\widetilde{e(\alpha x)f},\widetilde{e(\alpha x)g} \in C(\A/\mathbb{Q}, \mathbb{R}).
\]
Construct the following homotopy \( H \):
\begin{align*}
H: \mathbb{R} \times \A/\mathbb{Q} &\to S^1 \\
(t, \{x_\infty, x_f\}) &\mapsto exp(2\pi i \left( t \cdot \widetilde{e(\alpha x)f} + (1-t) \cdot \widetilde{e(\alpha x)g} \right)) \cdot e(\alpha x).
\end{align*}
Thus, \(H\) is continuous, \(H(0,x)=g(x), H(1,x)=f(x)\) and for any $t\in [0,1]$, $H(t,x)\in C(\A/\mathbb{Q}, S^1)$. 

On the other hand, if \( f \) and \( g \) are homotopic, \(f\) can be continuously transformed into \( g \), so 
\[ f \text{ and }g \text{ are homotopic}\Leftrightarrow\#f=\#g .\]

As a conclusion of this section, we obtain the following theorem:
\begin{theorem}\label{K1Q}
    Let $\A$ be the adele ring, $\mathcal{U}(\text{C}(\A/\mathbb{Q}))$ be the unitary element in $\text{C}(\A/\mathbb{Q})$, and  $\mathcal{U}_0(\text{C}(\A/\mathbb{Q}))$ be the unitary elements in $\text{C}(\A/\mathbb{Q})$ whose generalized winding numbers are zero. Then we have: 
    \[
    K_1(C^*(\mathbb{Q}))=K_1(C(\A/\mathbb{Q}))=\mathcal{U}(\text{C}(\A/\mathbb{Q}))/\mathcal{U}_0(\text{C}(\A/\mathbb{Q}))\cong\mathbb{Q}.
    \]
    \end{theorem}
    \begin{remark}\label{K0Q}
        Relatively, to compute the \( K_0 \)-group of \( C^*(\Q) \), it is necessary to consider the projections in \( C_0(\A/\Q) \), which are uniformly continuous \( \Q \)-periodic functions taking values in \( \{0, 1\} \).
        If $f(0)=1$, then for the continuity of $f$ we can choose a proper integer $N$ such that $f_{|\mathbb{R}\times K_N}\equiv 1$. Finally, for any $x=\{x_\infty,x_2,x_3,\dots\}\in \A$,  there exists $q\in \Q$ satisfying
        \[
        x-q=\{x_\infty-q,x_2-q,x_3-q,\dots\}\in \mathbb{R}\times K_N
        \]
by Lemma \ref{china}. Hence, the $\Q$-periodicity implies that $f\equiv 1$. We have obtained that there is no non-trivial projection in \( C_0(\A/\Q) \), so 
\[
K_0(C^*(\Q))=K_0(C_0(\A/\Q))=\mathbb{Z}.
\]
    \end{remark}

\section{$K_1$ group of $C^\ast(\mathbb{A})$}
We first consider the unitization of \( C_0(\mathbb{A}) \), denoted by \( C_0(\mathbb{A})^+\), endowed with the following algebraic operations:
\begin{align*}
    (f , \lambda) + (g , \mu) &= (f+g , \lambda+\mu), \\
    (f , \lambda)(g , \mu) &= (fg + \mu f + \lambda g, \lambda\mu),\\
    (f , \lambda)^*&=(f^* , \bar{\lambda}),
\end{align*}
where \(f,g \in C_0(\A), \lambda,\mu\in \mathbb{C}.\)

For a unitary element \((f, \lambda)\) in \(\CA\), it satisfies the characteristic property that
\( f + \lambda \) constitutes a continuous map from \(\A\) to the unit circle \(S^1\), 
with the limit at infinity being \(\lambda\):
\[
\lim_{\rvert x\lvert_{\A} \to \infty} (f(x) + \lambda) = \lambda \in S^1.
\]

We define a function in $C(\mathbb{R},S^1)$ as follow:
\[
f^\lambda_{x_f}(x_\infty):=f(\{x_\infty,x_f\}) + \lambda.
\]
The mapping \( f + \lambda \) is uniformly continuous on \( \A \). By selecting an appropriate positive integer \( N \), we ensure that for any \( x_f, y_f \in \A_f \) that satisfies \( x_f, y_f \in \alpha + K_N \) with some rational number \( \alpha \), the following holds:
\[
\sup_{x\in \mathbb{R}}\lvert f^\lambda_{x_f}-f^\lambda_{y_f}\rvert<\frac{1}{8}.
\]

By Remark \ref{rem}, there exists a lift \(\widetilde{f^\lambda_{x_f}}\) of \(f^\lambda_{x_f}(x_\infty)\). With Lemma \ref{china}, \(\A_f = \bigsqcup_{\alpha\in \mathbb{Q}/N\mathbb{Z}} \left( \alpha + K_N\right)\), 
we can construct lifts of \( f^\lambda_{x_f} \) on each subset \(\alpha+K_N\), analogous to the procedure developed in \sectwo, which yields a continuous mapping from \(\A\) to \(\mathbb{R}\):
\[
\widetilde{f+\lambda}:\ \{x_\infty,x_f\}\mapsto\widetilde{f^\lambda_{x_f}(x_\infty)},
\]
and
\[
exp(2\pi i\widetilde{f+\lambda})=f+\lambda.
\]
For any \((f+\lambda)\in C(\mathbb{R},S^1)\), we can define the following map:
\begin{align}
    \#(f+\lambda):&\A_f\mapsto \mathbb{Z}\\
                  &x_f\mapsto\#f^\lambda_{x_f},
\end{align}
it can be verified that this constitutes a homotopy invariant.

If $\#(f+\lambda)=\#(g+\mu)$, we can construct the following homotopy:

\begin{align*}
    H:[0,1]\times\A\mapsto&\A\\
         (t,\{x_\infty,x_f\})\mapsto& exp(2\pi i[(1-t)\cdot\widetilde{f+\lambda}(\{x_\infty,x_f\})\\
         &+t\cdot\widetilde{g+\mu}(\{x_\infty,x_f\}]).
\end{align*}

Therefore, we have the following theorem:
\begin{theorem}\label{K1A}
    Denote the unitary element in \(C_0(\A)^+\) by \(\mathcal{U}(C_0(\A)^+)\), and the unitary element homotopic to the constant function by \(\mathcal{U}_0(\CA)\), then \[K_1(C_0(\A)^+)=
    \mathcal{U}(C_0(\A)^+)/\mathcal{U}_0(\CA)=C_0(\A_f,\mathbb{Z}).
    \]
Moreover, we have the exact sequence expressed as:
\[ 0\rightarrow K_1(C_0(\A))\rightarrow K_1(\CA^+)\rightarrow K_1(\mathbb{C})\rightarrow 0, \] and
we arrive at the result:
\[
K_1(C^*(\A))=K_1(\CA)=C_0(\A_f,\mathbb{Z}).
\]
\begin{remark}\label{K0A}
To consider the projection $g$ in $\CA^+$, we first claim that $g_{x_f}(x_\infty):=g(\{x_\infty,x_f\})$ is uniformly  continuous function in $C(\mathbb{R})$ for any $x_f\in \A_f$, and $g_{x_f}$ has a common limit of 0 or 1 as $x_\infty$ tends to infinity. Therefore, $ g $ is identically equal to the constant value 0 or 1, which implies that
\[
K_0(\CA^+)=K_0(\mathbb{C})=\mathbb{Z}.
\]
We have $K_0(C_0(\A))=0$.    
\end{remark}
\end{theorem}
\newpage
\appendix
\section{Metric construction on the adele ring \texorpdfstring{$\A$}{A}}\label{metric}

The adele ring $\mathbb{A}$ of $\mathbb{Q}$ is endowed with the \textbf{restricted direct product topology}, explicitly constructed as:

\begin{equation*}
\mathbb{A}= {\prod_{p \leq \infty}}'\left(\mathbb{Q}_p, \mathbb{Z}_p\right) := 
\left\{ 
    (x_p) \in \prod_{p \leq \infty} \mathbb{Q}_p \,\bigg|\, 
    \begin{aligned}
        &x_p \in \mathbb{Z}_p \text{ for } \\
        &\text{almost all primes } p 
    \end{aligned}
\right\}
\end{equation*}
where the topology is generated by the fundamental system of open neighborhoods:
\begin{equation*}
U = \left( \prod_{p \in S} U_p \right) \times \left( \prod_{p \notin S} \mathbb{Z}_p \right)
\end{equation*}
with:
\begin{itemize}
    \item $S = \{\infty\} \cup \{\text{finite set of primes}\}$
    \item $U_p \subseteq \mathbb{Q}_p$ open in the standard topology ($\mathbb{R}$ when $p=\infty$)
    \item $\mathbb{Z}_p$ denoting the $p$-adic integers for primes $p < \infty$
\end{itemize}

This specific construction for $\mathbb{Q}$ maintains the key property that $\mathbb{A}$ is a locally compact topological ring.

We claim that the following map from $\A$ to $\mathbb{R}$ is a metric:
\begin{align}
    \lvert\cdot\rvert_{\A}:\{x_\infty,x_f\}\mapsto \lvert x_\infty\rvert_{\infty}+\lvert x_f\rvert_{\A_f},
\end{align}
where
\[
|x_f|_{\A_f} = 
\begin{cases} 
\max \left( \frac{|x_p|_p}{p} \right), & \forall p, x_p \in \mathbb{Z}_p \\
\max |x_p|_p, & \text{otherwise}.
\end{cases}
\]
Since both $\lvert\cdot\rvert_\infty$ and $\lvert\cdot\rvert_{p}$ satisfy the properties of non-negativity, symmetry, and the triangle inequality, it follows that\(\lvert\cdot\rvert_{\A}\) also constitutes a metric. 
\begin{theorem}
    The metric-induced topology and the original restricted product topology on the adele ring $\A$ are equivalent.

\end{theorem}
\begin{proof}
To establish the equivalence between the metric-induced topology and the original restricted product topology on the adele ring $\A$, we must rigorously verify the following two conditions: 
\begin{enumerate}
    \item For every metric ball \( B_\varepsilon(x) := \{ y \in \mathbb{A}_{\mathbb{Q}} \mid d(x,y) < \varepsilon \} \) and any \( x \in B_\varepsilon(x) \), there exists a basic open set \( U \) in the restricted product topology such that:
    \[
    x \in U \subseteq B_\varepsilon(x).
    \]
    
    \item For every open set \( U \) in the restricted product topology and any \( x \in U \), there exists \( \delta > 0 \) such that:
    \[
    B_\delta(x) \subseteq U.
    \]
\end{enumerate}
    Case 1: We define the following open set:
    \begin{equation}
        U_0:=(-\frac{\varepsilon}{2},\frac{\varepsilon}{2})\times\prod_{p\leq\frac{2}{\varepsilon}}p^{log_2\lfloor\frac{2}{\varepsilon}\rfloor+1}\mathbb{Z}_p\times\prod_{p>\frac{2}{\varepsilon}}\mathbb{Z}_p,
    \end{equation}
    If $y=\{y_\infty,y_2,y_3,\dots\}\in U_0$, it can be verified that:
\[  
\begin{aligned}
    &\lvert y_\infty\rvert_\infty<\frac{\varepsilon}{2},        && \text{infinite place} \\
    &\frac{\lvert y_p\rvert_p}{p}\leq \frac{1}{p^{log_2\lfloor\frac{2}{\varepsilon}\rfloor+1}}<\frac{1}{2^{log_2\frac{2}{\varepsilon}}}=\frac{\varepsilon}{2},           && \forall p\leq\frac{2}{\varepsilon} \\
    &\frac{\lvert y_p\rvert_p}{p}<\frac{1}{2/\varepsilon}=\frac{\varepsilon}{2},           &&  \forall p>\frac{2}{\varepsilon}.
\end{aligned}
\]  
    Hence, \(\lvert y\rvert_{\A}<\frac{\varepsilon}{2}+\frac{\varepsilon}{2}=\varepsilon\), $x\in x+U_0\subset B_{\varepsilon}$.
    
    Case 2: Let
    \[x=\{x_\infty,x_2,x_3,x_5,\dots\}\in U = \left( \prod_{p \in S} U_p \right) \times \left( \prod_{p \notin S} \mathbb{Z}_p \right).\]
    For any place $p\in S$, if $p\neq\infty$, we can find a positive integer $m_p$ such that $p^{m_p}\mathbb{Z}_p\in U_p$, let
    \[
    I=\inf_{p\in S,\ p\neq \infty}\{p^{-{(m_p+1)}}\}.
    \]
    
    Choose a positive number $\frac{1}{2}>\varepsilon>0$ in order to ensure:
    \[
    x+(-\varepsilon,\varepsilon)\times\prod_{p\in S,\ p\neq \infty}p^{m_p}\mathbb{Z}_p\times\prod_{p \notin S} \mathbb{Z}_p \subseteq U,
    \]
    then define \(\delta=\frac{1}{2}\min\{\varepsilon, I\}\). Consider $y=\{y_\infty,y_2,y_3,\dots\}\in B_{\delta}(0)$, and conclude that:
    \begin{align*}
          &\lvert y_\infty\rvert_\infty<\frac{\varepsilon}{2},        && \text{infinite place} \\
    &\frac{\lvert y_p\rvert_p}{p}\leq \delta\Rightarrow \lvert y_p\rvert_p<p^{-m_p}\Rightarrow y_p\in p^{m_p}\mathbb{Z}_p,           && \forall p\in S\\
    &\frac{\lvert y_p\rvert_p}{p}<\frac{1}{2}\Rightarrow\lvert y_p\rvert_p <p\Rightarrow y_p\in\mathbb{Z}_p,           &&  \forall p\notin S,
    \end{align*}
    which implies
    \[
    B_{\delta}(x)\subseteq x+\left((-\varepsilon,\varepsilon)\times\prod_{p\in S,\ p\neq \infty}p^{m_p}\mathbb{Z}_p\times\prod_{p \notin S} \mathbb{Z}_p \right)\subseteq U.
    \]
    In summary, we have demonstrated the equivalence of the two topologies. 
\end{proof}
\printbibliography 
\end{document}